\documentclass[english,11pt]{article}

\usepackage{amsthm}
\usepackage[utf8]{inputenc}
\usepackage{babel}
\usepackage[hmargin=0.9in]{geometry}

\usepackage{algorithm}
\usepackage{algpseudocode}

\usepackage{graphicx}
\usepackage{xcolor}

\usepackage{diagbox}
\usepackage{multirow}

\usepackage[colorlinks=true,citecolor=blue,urlcolor=blue]{hyperref}
\usepackage{url}

\usepackage{amsmath}
\usepackage{amssymb,amsfonts}
\usepackage{upgreek}

\newtheorem{remark}{Remark}

\newtheorem{theorem}{Theorem}[section]
\newtheorem{corollary}{Corollary}[theorem]
\newtheorem{lemma}[theorem]{Lemma}



\def\bfq{ {\bf q} }
\def\bfql{ {\bf q}_\ell }
\def\bfqr{ {\bf q}_r }

\def\bff{ {\bf f} }
\def\bfg{ {\bf g} }

\def\bfw{ {\bf w} }

\def\bfZero{ {\bf 0} }

\newcommand{\po}{p^{[0]}}
\newcommand{\ho}{h^{[0]}}
\newcommand{\hot}{\tilde{h}^{[0]}}
\newcommand{\hk}{h^{[k]}}
\newcommand{\xk}{x^{[k]}}

\newcommand{\real}{\mathbb{R}}

\newcommand{\hmin}{h_{\min}}
\newcommand{\hmax}{h_{\max}}
\newcommand{\pmin}{p_{\min}}
\newcommand{\pmax}{p_{\max}}
\newcommand{\lr}{{\ell \! / \! r}}
\renewcommand{\SS}{\mathrm{SS}}
\newcommand{\PV}{\mathrm{PV}}
\newcommand{\CC}{\mathrm{CC}}
\newcommand{\RR}{\mathrm{RR}}
\newcommand{\AV}{\mathrm{AV}}
\newcommand{\QA}{\mathrm{QA}}

\newcommand{\revX}[1]{#1}  
\newcommand{\revY}[1]{#1}  

\usepackage{authblk}
\author[1]{Carlos Muñoz-Moncayo\thanks{carlos.munozmoncayo@kaust.edu.sa}}
\author[1]{Manuel Quezada de Luna\thanks{manuel.quezada@kaust.edu.sa}}
\author[1]{David I. Ketcheson\thanks{david.ketcheson@kaust.edu.sa}}

\affil[1]{King Abdullah University of Science and Technology (KAUST), Thuwal 23955-6900, Saudi Arabia}

\begin{document}
\title{A Comparative Study of Iterative Riemann Solvers for the Shallow Water and Euler Equations}

\date{\today}
\maketitle

\abstract{
We investigate the achievable efficiency of exact solvers for the Riemann
problem for two systems of first-order hyperbolic PDEs: the shallow water equations
and the Euler equations of compressible gas dynamics.
Many approximate solvers have been developed for these systems;
exact solution algorithms have received less attention because \revX{the } computation of the
exact solution typically requires \revX{an }  iterative solution of algebraic
equations, which can be expensive or unreliable.
We investigate a range of iterative algorithms and initial guesses.
In addition to existing algorithms, we propose simple new algorithms
that are guaranteed to converge and to remain in the range of physically
admissible values at all iterations.
We apply the existing and new iterative schemes to an ensemble of test Riemann problems.
For the shallow water equations, we find
that Newton's method with a simple modification converges quickly and reliably.
For the Euler equations we obtain similar results; however, when the
required precision is high, a combination of Ostrowski and Newton
iterations converges faster.  These solvers
are slower than standard approximate solvers like Roe and HLLE, but come
within a factor of two in speed.  We also provide a preliminary comparison
of the accuracy of a finite volume discretization using an exact solver
versus standard approximate solvers.

}

\section{Background and Motivation}

The Riemann problem is central to both the theory and
the numerical solution of hyperbolic partial differential equations (PDEs).
Most modern numerical methods for hyperbolic problems --
including finite difference, finite volume, and finite element discretizations --
are based on approximating the solution of the Riemann problem.
In the course of a simulation, billions of Riemann problems may be solved 
at each step, so fast solvers are essential.  However, the exact
solution of the Riemann problem typically requires the solution
of a nonlinear algebraic system, which may be costly.  It is principally for
this reason that approximate Riemann solvers have been preferred,
and exact Riemann solvers are rarely, if ever, employed.  In this work
we use the term \emph{exact} to refer to a solver that guarantees a solution
accurate to within a prescribed tolerance, such as machine precision.

We focus on what are perhaps the most-studied hyperbolic
systems: the shallow water (or Saint-Venant) equations and the compressible
Euler equations.  These two systems have a similar structure and have played
a key role in the development of numerical methods.  A wide variety of
approximate Riemann solvers is available for each of them; see \cite{toro2001shock,toro2013riemann}
for an extensive exposition.  The exact solution
for either of these systems can be determined by solving a scalar nonlinear
algebraic equation.  The two major perceived drawbacks of doing so are
the cost of an iterative algebraic solver and the lack of robustness
(i.e., the lack of a guarantee of convergence).

Although approximate Riemann solvers have long since emerged as preferable
to exact solvers for numerical computation, there is a small body of work devoted
to developing fast exact Riemann solvers for the Euler equations of gas
dynamics; the most advanced methods available were developed about 30 years ago 
\cite{gottlieb1988assessment,pike1993riemann}.
We mention also one very recent project in which the exact solution was obtained by training
a neural network \cite{gyrya2020machine}; however, evaluating the neural network
was significantly slower than applying traditional approximate solvers.

In simulation, exact solvers can be employed adaptively, to be used in regions
of more extreme conditions, while approximate solvers are used in other regions
\cite{toro1991,toro1993linearized,quirk1994alternative}.
Exact solvers are also of independent interest, since they are required in the
random choice method \cite{glimm1965solutions} and may be used for other theoretical purposes.
Another potential application of the solvers discussed in this work are the ADER 
methods, which are designed to attain arbitrarily high order of accuracy for finite volume and finite element 
solution of balance laws, and require an approximation of the exact Riemann solution for the homogeneous system
as an essential step; see e.g. \cite{titarev2002ader,toro2020ader,helzel2019ader}.
Furthermore, it seems worthwhile to revisit the question of whether exact Riemann
solvers may be of practical use, for three reasons.  First, relative changes
in the computational cost of flops versus memory access over the last two
decades may mean that iterative solvers are now more competitive.  Second,
some new ideas have appeared in the literature recently regarding potential
initial guesses and iterative methods.
Third, many codes that make use of Riemann solvers also have other expensive
computational parts; the cost of the iterative solver may not be dominant
relative to these other components.  If the cost is not prohibitive, a
robust solver that guarantees a solution with a desired accuracy and required properties is attractive.

We survey and test established methods, while also
extending and improving them in various ways.  Whereas
previous studies \cite{gottlieb1988assessment,pike1993riemann} 
focused exclusively on the Euler equations, we also study the shallow water equations,
which have a similar structure.  We consider a wider range of
initial guesses and iterative methods.
We also address some questions related to convergence criteria that prevented
previous approaches from giving accurate results in certain regimes.

The specific contributions of the present work are:
\begin{itemize}
    \item the introduction and testing of some new initial guesses and iterative solvers;
    \item the first methods, for both the Euler and shallow water equations,
        that are provably-convergent for all initial data; these are also proven to
        yield a physically-admissible state at every iteration (see Section
        \ref{sec:newton}).  For the shallow water equations, this
        is based on a new proven upper bound for the depth (see Section \ref{IG: RR}).
    \item the most extensive comparison of iterative Riemann solvers to date;
    \item an updated comparison of the relative computational cost of
        approximate versus iterative solvers, using typical modern computational hardware
        (see Section \ref{sec:FV}).
\end{itemize}

The remainder of this work is organized as follows.  In Section \ref{sec:sw}
we discuss the shallow water equations, reviewing the structure of the Riemann
solution.  We review many existing initial guesses and iterative techniques for
its solution, and present some new ones.
A comparison of all of these methods is given in Section \ref{sec:Results SWEs}.
In Section \ref{sec:Euler} we provide a similar analysis and evaluation of
iterative solvers for the Riemann problem for the Euler equations.
In Section \ref{sec:FV}, we conduct a preliminary comparison of the best
approaches found in Sections \ref{sec:sw} and \ref{sec:Euler} with standard
approximate Riemann solvers in finite volume simulations .  

The code used for this work is available in the reproducibility repository \cite{repository}.

\section{The shallow water equations}\label{sec:sw}

In two space dimensions, the shallow water equations take the form
  \begin{align}
    \partial_t \bfq + \partial_x \bff(\bfq) + \partial_y \bfg(\bfq) = \bfZero,
  \end{align}
where
  \begin{align}
    \bfq = \begin{bmatrix}
      h \\
      hu \\
      hv
    \end{bmatrix}, \quad
    \bff(\bfq) = \begin{bmatrix}
      hu \\
      hu^2+\frac{1}{2}gh^2 \\
      huv
    \end{bmatrix}, \quad
    \bfg(\bfq) = \begin{bmatrix}
      hv \\
      huv \\
      hv^2+\frac{1}{2}gh^2 \\
    \end{bmatrix}.
   \end{align}
In this section we briefly review the exact solution of this problem and then
describe several iterative algorithms for its approximation.

    \subsection{Exact solution of the Riemann problem}
    For simplicity we consider a Riemann problem with coordinate-aligned initial data.
    The solution for any other plane-wave Riemann problem can be obtained from what
    follows by simple rotation.
    The structure of the Riemann solution is then the same as for the
    one-dimensional shallow water equations with a passive tracer:
    \begin{subequations}
    \label{eq: RP SW_Eqns}
  \begin{align}
    \partial_t \bfq + \partial_x \bff(\bfq) = \bfZero, \quad \forall (x,t)\in\mathbb{R}\times\mathbb{R}^+, \label{eq: RP SW_Eqns eq}\\
    \bfq(x,0)=\begin{cases}\label{eq: RP SW_Eqns IC}
    \bfq_\ell, \quad \text{ if } x<0, \\
    \bfq_r, \quad \text{ if } x>0.
    \end{cases}
  \end{align} 
    \end{subequations}
    Here we review the solution structure; a more detailed discussion can be found in several extensive works
   \cite{toro2001shock,levequefvmbook,ketcheson2020riemann}.
    It is convenient to make use of the vector of primitive variables $\bfw = [h, u, v]^T$.
    The solution involves the initial states $\bfw_\ell$ and $\bfw_r$, given by the initial data 
    \eqref{eq: RP SW_Eqns IC}, as well as two intermediate states $\bfw_{*\ell}$, $\bfw_{*r}$.
    These states are connected by three waves.  We would like to determine the middle states
    $\bfw_{*\ell}=(h_*,u_*,v_{*\ell})^T$, $\bfw_{*r}=(h_*,u_*,v_{*r})^T$ and the
    structure of the waves connecting them. 
    The characteristic speeds of the system \eqref{eq: RP SW_Eqns eq} are given by
    the eigenvalues of the flux Jacobian $A(\bfq)=\bff^\prime(\bfq)$:
    \begin{align}
      \lambda_1=u-\sqrt{gh}\leq\lambda_2=u\leq\lambda_3=u+\sqrt{gh}.
    \end{align}
    We will refer to the wave associated with a given eigenvalue by the corresponding
    index. 
    For this problem, the 2-wave is a linearly degenerate contact wave (or contact discontinuity).
    The values of $h$ and $u$ are constant across the 2-wave, so we write simply
    $h_*=h_{*\ell}=h_{*r}$ and $u_*=u_{*\ell}=u_{*r}$.
    The 1- and 3-waves are genuinely nonlinear.  The 1-wave is a shock if $h_*>h_\ell$,
    in which case the states on either side of the 1-wave are related by the Rankine-Hugoniot
    condition.  On the other hand, if $h_*<h_\ell$, the 1-wave is a rarefaction, in
    which case the states on either side are related through the corresponding Riemann
    invariant.  Similar statements apply to the 3-wave and the relation between $h_*$ and $h_r$.
    Taking these conditions together, one finds that $h_*$ is given by the unique real root of 
    \begin{subequations}
    \label{eq:phi SWEs}
    \begin{align}
      \phi(h;\bfq_\ell,\bfq_r) &= f(h;h_\ell) + f(h;h_r) + u_r - u_\ell, 
    \end{align}
    where
    \begin{align}	
    \label{eq:f SWEs}
      f(h;h_\lr) &= \begin{cases}
        f_R(h;h_\lr) := 2(\sqrt{gh}-\sqrt{gh_\lr}), 
        & \text{ if } h\leq h_\lr \text{ (rarefaction)}, \\
        f_S(h;h_\lr) := (h-h_\lr)\sqrt{\frac{1}{2}g\left(\frac{h+h_\lr}{h h_\lr}\right)}, 
        & \text{ if } h>h_\lr \text{ (shock)}.
      \end{cases}
    \end{align}
    \end{subequations}
The function $\phi$ is sometimes referred to as the depth function
\cite{toro2001shock}.  Since we are interested in finding a
root of $\phi$ for (fixed but arbitrary) given states $\bfq_\ell$ and $\bfq_r$,
we will write simply $\phi(h)$ for brevity.
Once $h_*$ is found, the corresponding velocity is given by 
\begin{align}
 \label{eq: speed star SWEs}
  u_*=\frac{1}{2}(u_\ell+u_r)+\frac{1}{2}[f(h_*;h_r)-f(h_*;h_\ell)].
\end{align}

In this work we focus on Riemann problems for which the solution satisfies $h>0$ at all points.
Riemann problems in which dry states appear or are initially present can be detected \emph{a priori}
and solved by explicit formulas.
\revY{For practical use, an implementation of the dry-state treatment presented in \cite[Sec. 6.3]{toro2001shock}
is provided in the reproducibility repository for this work \cite{repository}.}

\subsection{Properties of the depth function}
\label{sec: properties depth SWE}
For $h>0$, the first three derivatives of $f(h,h_\lr)$ are given by
\begin{subequations}
\begin{align}
  f'(h;h_\lr) = \begin{cases}
    \frac{\sqrt{g}}{\sqrt{h}}, & \text{ if } 0<h\leq h_\lr, \\
    \frac{g(2h^2+hh_\lr+h_\lr^2)}{2\sqrt{2}h^2\sqrt{g\left(\frac{1}{h}+\frac{1}{h_\lr}\right)}h_\lr}
    & \text{ if } h>h_\lr>0, 
  \end{cases}, 
\end{align}
\begin{align}
  f''(h;h_\lr) = \begin{cases}
    \frac{-\sqrt{g}}{2\sqrt{h^3}}, & \text{ if } 0<h\leq h_\lr, \\
    \frac{-g^2(5h+3h_\lr)}{4\sqrt{2}h^4\sqrt{g^3\left(\frac{1}{h}+\frac{1}{h_\lr}\right)^3}},
    & \text{ if } h>h_\lr>0,
  \end{cases}, 
\end{align}
and
\begin{align}			
  f'''(h;h_\lr) = \begin{cases}
    \frac{3\sqrt{g}}{4\sqrt{h^5}}, & \text{ if } 0<h\leq h_\lr, \\
    \frac{3gh_\lr(10h^2+13hh_\lr+5h_\lr^2)}{8\sqrt{2}h^4\sqrt{g\left(\frac{1}{h}+\frac{1}{h_\lr}\right)}(h+h_\lr)^2}
    & \text{ if } h>h_\lr>0,
  \end{cases}.
\end{align}
\end{subequations}
Observe that (for $h>0$) $\phi'(h)=f'(h;h_\ell)+f'(h;h_r)>0$ regardless of whether the
nonlinear waves are shocks or rarefactions.  Similarly, (for $h>0$) $\phi''(h)<0$ and
$\phi'''(h)>0$  \cite{toro2001shock}.  Finally, observe that $\phi(0)\le 0$ and
$\lim_{h\to\infty} \phi(h) = +\infty$.  Together, these properties guarantee
that $\phi(h)$ has a unique real root. Also, letting
\begin{subequations} \label{eq:hminmax}
\begin{align}
    \hmin & := \min \left(h_{\ell}, h_{r}\right), \\
    \hmax & := \max \left(h_{\ell}, h_{r}\right),
\end{align}
\end{subequations}
we have that if $\phi\left(\hmin\right)>0$ both waves are rarefactions, and if
$\phi\left(\hmax\right)<0$ both waves are shocks.

\subsection{Computation of the initial guess}\label{subsec: IG}
    
Since we will proceed iteratively to approximate $h_*$, it is advantageous to
start with an accurate initial guess $\ho$ that can be computed efficiently.
Moreover, as we will see, a lower bound $0<h_{-}\leq h_*$ will also be required
for the iterative part of most algorithms.
We describe seven possible initial guesses in the following sections.

\subsubsection{Two-rarefaction approximation (RR)}\label{IG: RR}
 If we assume that the 1- and 3- waves are rarefactions, then the root of $\phi$ can be computed explicitly:
 \begin{equation}
 \label{eq:IG RR}
 h_\RR := \frac{\left(u_{\ell}-u_{r}+2 \sqrt{g h_{\ell}}+2 \sqrt{g h_{r}}\right)^{2}}{16 g}.
 \end{equation}
 In case one or both of the waves is a shock, this is not the exact solution
 but can serve as an initial guess.  This estimate never underestimates
 the true depth, as we now prove (see a similar result
 obtained by Guermond \& Popov for the Euler equations \cite[Section 4.1]{guermond2016fast}).
 \begin{lemma} \label{lemma:upper bound two rarefaction}
 $h_*\leq h_\RR$.
 \end{lemma}
 \begin{proof}

 Let
\begin{align}
  \phi_\RR(h) := 2(\sqrt{gh}-\sqrt{gh_\ell})+2(\sqrt{gh}-\sqrt{gh_r})+u_r-u_\ell
  = f_R(h,h_\ell)+f_R(h,h_r)+u_r-u_\ell,
\end{align}
which is the corresponding $\phi(h)$ function if both waves were rarefactions.
Let us first prove that 
\begin{equation}
h>h_\lr\quad \Rightarrow \quad f_R(h,h_\lr)< f_S(h,h_\lr).
\end{equation}
Rewrite $f_R(h,h_\lr)$ and $f_S(h,h_\lr)$ as
\begin{align*}
  f_R = 2\sqrt{gh_\lr}\left(\frac{x-1}{\sqrt{x}+1}\right), \qquad
  f_S = \frac{\sqrt{gh_\lr}}{\sqrt{2}}\left(\frac{(x-1)\sqrt{x+1}}{\sqrt{x}}\right), 
\end{align*}
where $x:=\frac{h}{h_\lr}$. We want to show that $f_R(x)< f_S(x)$ when $x> 1$. 
To see this, rewrite $f_S(x)$ as 
\begin{align*}
  f_S(x) = \underbrace{\frac{\sqrt{x+1}(\sqrt{x}+1)}{2\sqrt{2}\sqrt{x}}}_{=:g(x)}
  \left[2\sqrt{gh_\lr}\left(\frac{x-1}{\sqrt{x}+1}\right)\right] 
  = g(x)f_R(x).
\end{align*}
The first derivative of $g(x)$ is 
\begin{align*}
  g'(x) = \frac{\sqrt{x^3}-1}{4\sqrt{2x^3(x+1)}}.
\end{align*}
Finally, note that $g(1)=1$ and $g'(x)>0, \forall x>1$.
Therefore, $g(x)>1, \forall x>1 \implies f_R(x)< f_S(x), \forall x>1$. If both waves are rarefactions there is nothing to prove. If both waves are shocks, then for every $h>h_{\max}$
\begin{align*}
 \phi(h)	&= f_S(h,h_{\min})+f_S(h,h_{\max})+u_r-u_\ell\\
								&> f_R(h,h_{\min})+f_R(h,h_{\max})+u_r-u_\ell\\
								&=\phi_\RR(h).
\end{align*}
Therefore $h_*<h_\RR$. If there is one rarefaction and one shock, we proceed analogously.

 \end{proof}
\subsubsection{Average depth (AV)}\label{IG: AV}
A simple and inexpensive initial guess is the arithmetic average:
\begin{equation} \label{eq:IG AV}
\ho = h_\AV := \frac{h_\ell+h_r}{2}.
\end{equation}

 \subsubsection{Quadratic approximation (QA)}\label{IG: QA}
We also consider an initial guess presented in Lemma 3.8 of the recent work by Azerad et. al.\cite{azerad2017well}, given by
\begin{align} \label{eq:IG QA}
\ho = {h}_\QA:=\left\{\begin{array}{ll}
h_\RR& \text { if } 0\leq \phi(x_0 h_{\min}), \\
\sqrt{h_{\min } h_{\max }}\left(1+\frac{\sqrt{2}\left(u_{\ell}-u_{r}\right)}{\sqrt{g h_{\min }}+\sqrt{g h_{\max }}}\right) & \text { if } \phi(x_0 h_{\max})<0, \\
\left(-\sqrt{2 h_{\min }}+\sqrt{3 h_{\min }+2 \sqrt{2 h_{\min } h_{\max }}+\sqrt{\frac{2}{g}}\left(u_{\ell}-u_{r}\right) \sqrt{h_{\min }}}\right)^{2} & \text { otherwise } ,
\end{array}\right.
\end{align}
where $x_{0}=(2 \sqrt{2}-1)^{2}$. It also satisfies the property $h_\QA \ge h_*$.

\subsubsection{Primitive-variable linearized solver (PV)}
Any approximate-state Riemann solver can be used to provide an initial guess for
an iterative solver.  Here we consider the use of the linearized solver
introduced by Toro \cite{toro2001shock}.  Linearizing the system
around the state $\hat{\bfw}$ in primitive variables gives
\begin{equation}
\label{eq:linearized system}
\bfw_t+\bf{A}(\hat{\bfw})\bfw_x=0.
\end{equation}
Solving this system exactly we get
\begin{align}
\hat{h}=\frac12(h_\ell+h_r)+\frac{\frac12(u_\ell-u_r)}{\sqrt{2g/(h_\ell+h_r)}}.
\end{align}
We consider a modification of this solution, constructed to enforce positivity preservation \cite{toro2001shock}
\begin{align}\label{eq:IG PV SWEs}
h^{[0]}=h_\PV := \frac12(h_\ell+h_r)+\frac{\frac14(u_\ell-u_r)(h_\ell+h_r)}{\sqrt{gh_\ell}+\sqrt{gh_r}}.
\end{align}

\subsubsection{Two-shock approximation (SS)} \label{IG:SS SWEs}





If we assume that both waves are shocks, then the depth function \eqref{eq:phi SWEs} becomes
\begin{equation} \label{eq:phi SS SWEs}
\phi_\SS(h)=(h-h_\ell)\sqrt{\frac12 g \left( \frac{h+h_\ell}{h\,h_\ell}\right)}+(h-h_r)\sqrt{\frac12 g \left( \frac{h+h_r}{h\,h_r}\right)} + u_r-u_\ell.
\end{equation} 
Since we cannot find a root of this function explicitly, we
approximate it linearly as is done by Toro \cite{toro2001shock}.
We set $y_\lr(h)=\sqrt{\frac12 g \left(
\frac{h+h_\lr}{h\,h_\lr}\right)}$ and assume that we have an estimate of the root,
denoted by $\bar{h}$. Then our problem becomes
\begin{equation}
(h-h_\ell)y_\ell(\bar{h})+(h-h_r) y_r(\bar{h})+u_r-u_\ell=0,
\end{equation}
and it follows that
\begin{equation}\label{eq:IG SS SWEs}
\ho = h_\SS := \frac{h_\ell y_\ell(\bar{h})+h_r y_r(\bar{h})-u_r+u_\ell}{y_\ell(\bar{h})+y_r(\bar{h})}.
\end{equation}
Here we use the estimate $\bar{h}=h_\PV$.

\subsubsection{Convex combination} \label{IG: CC}
If $\phi(\hmin)>0$, the solution has two rarefactions and $h_*$ is given explicitly
by \eqref{eq:IG RR}.  Otherwise, we define the lower bound
\begin{equation} \label{eq: Initial quadratic left}
h_-=\begin{cases}
    \hmax & \text { if } \phi(\hmax)<0 \qquad\qquad\qquad \text { (1- and 3-shocks), } \\
    \hmin & \text { if } \phi(\hmin)<0<\phi(\hmax) \qquad \text { (1- (3-) shock and 3- (1-) rarefaction), }
    \end{cases}
\end{equation}
and the upper bound
\begin{equation} \label{eq: Initial quadratic right}
h_+=\begin{cases}
 h_\RR& \text { if } \phi(h_{max})<0 \qquad\qquad\qquad   \text { (1- and 3-shocks), } \\
\min(h_{\max},h_\RR) & \text { if } \phi(h_{min})<0<\phi(h_{max})\qquad  \text { (1- (3-) shock and 3- (1-) rarefaction), }
\end{cases}
\end{equation}
for $h_*$.

\revY{We can compute an initial guess that is a convex combination of these bounds,
choosing the coefficients such that the initial guess is biased towards the bound whose residual is closer to 0,}
\begin{align} \label{eq:IG CC}
\ho = h_\CC := \frac{|\phi(h_{+})|}{|\phi(h_{-})|+|\phi(h_{+})|}\, h_{-}+\frac{|\phi(h_{-})|}{|\phi(h_{-})|+|\phi(h_{+})|}\, h_{+}=\frac{\phi(h_{+})h_{-}-\phi(h_{-})h_{+}}{\phi(h_{+})-\phi(h_{-})}.
\end{align}
\revY{This initial guess can be further interpreted as the root of an affine function that intersects with $\phi$ at $h_-$ and $h_+$.}

\subsubsection{HLLE} \label{IG:HLLE}
We also consider the HLLE solver, first introduced by Harten, Lax, and van Leer, and then improved by Einfeldt \cite{einfeldt1988godunov}.
It is based on the assumption that the solution consists of three constant states, $q_\ell$, $q_m$, and $q_r$, separated by two waves.
Assuming that the speeds of these waves are $s_\ell$ and $s_r$ and, by imposing conservation on the approximate solution, we get 
\begin{align}\label{eq:IG HLLE}
    q_m=\frac{f(q_r)-f(q_\ell)-s_r\,q_r+s_\ell\,q_\ell}{s_\ell-s_r}.
\end{align}
Here $s_\ell$ and $s_r$ are chosen in terms of the eigenvalues of $f'$ evaluated at the Roe average \cite{levequefvmbook}. Thus, we can set $\ho=h_{\mathrm{HLLE}}:=h_m$.

    \subsection{Iterative algorithms} \label{subsec: Iterative algorithms}
There are several approaches that can be taken for the iterative process. In
this section we describe the algorithms tested in this study.

        \subsubsection{Positive Newton method} \label{sec:newton}
        
Perhaps the most widely used root-finding method is the Newton-Rhapson iteration,
also known as Newton's method.  For a function $\psi(x)$, this method starts with
an initial guess $x^{[0]}$ and proceeds by the iteration
\begin{align} \label{eq:newton}
    x^{[k+1]} = x^{[k]} - \frac{\psi(x^{[k]})}{\psi'(x^{[k]})}.
\end{align}
In general, Newton's method is not guaranteed to converge.  But for
functions sharing some of the properties of $\phi(h)$, convergence is
guaranteed by starting from the left.

\begin{lemma} \label{lem:newton}
    Let $\psi(x) : (0,\infty)\to \real$ be a real-valued, twice
    continuously-differentiable function such that $\lim_{x\to 0}\psi(x)<0$,
    $\psi'(x)>0$, and $\psi''(x)<0$.  Then $\psi$ has a unique real root $x_*>0$.
    Let $x^{[0]}\le x^*$.  Then the sequence ${x^{[k]}}$ produced by Newton's
    method \eqref{eq:newton} has the properties that $x^{[k]}\in[x_0,x_*)$ for
    all $k$ and $x^{[k]}\to x_*$ as $k\to\infty$.
\end{lemma}
\begin{proof}
    The existence of a unique real root follows from $\lim_{x\to 0}\psi(x)<0$ and
    $\psi'(x)>0$.

    To prove the next property, assume by induction that $x^{[j]}\in[x_0,x_*)$
    for $j=0,\dots,k$.  Then the properties of $\psi$ guarantee that $\psi(x^{[k]})<0$
    and $\psi'(x^{[k]}) \ge \psi'(x_*)>0$.  This implies that $x^{[k]} < x^{[k+1]} < x_*$.

    Since the sequence of iterates is strictly increasing and bounded from above,
    it converges to some value  $\hat{x}$ and is a Cauchy sequence, i.e.
    \begin{align} \label{eq:cauchy}
        \lim_{k\to\infty} (x^{[k+1]}-x^{[k]}) = 0.
    \end{align}
    Since $\psi'$ is continuous, it is bounded in $[x_0,x_*]$  By this, the continuity
    of $\phi$, \eqref{eq:cauchy} and \eqref{eq:newton}, we have
    $$
    0 = \lim_{k\to\infty} \psi'(x^{[k]})(x^{[k+1]}-x^{[k]}) = \lim_{k\to\infty}\psi(x^{[k]}) = \psi(\hat{x}).
    $$
    But since $\psi$ has a unique real root, we must have $\hat{x}=x_*$.
\end{proof}

\begin{corollary}\label{cor:newton}
    Let left and right states $\bfq_\ell, \bfq_r$ be given and consider the application
    of Newton's method \eqref{eq:newton} to the depth function \eqref{eq:phi SWEs}.  If the initial
    guess satisfies $0\le h^{[0]}\le h_*$, then the iteration converges to $h_*$.
\end{corollary}
\begin{proof}
    If $h_*=0$, then $h^{[0]}=h_*$.  Otherwise, we can apply Lemma \ref{lem:newton}.
\end{proof}

In order to define an initial guess that satisfies Corollary \ref{cor:newton}, we replace our initial
guess $\ho$ with (see \cite{guermond2016fast}, where the analogous quantity is used for the Euler
equations)
\begin{equation}
\label{eq:IG correction Newton}
\hot := \max \left(h_{\min}, \ho -\phi\left(\ho \right) / \phi^{\prime}\left(\ho \right)\right).
\end{equation}
\revX{
Indeed, since $\phi$ is concave and strictly increasing, we have
\begin{gather}
    0=\phi(h_*) \leq \phi(\ho) + \phi'(\ho)(h_*-\ho),\\
    \ho -\phi\left(\ho \right) / \phi^{\prime}\left(\ho \right) \leq h_*.
\end{gather}
Moreover, if $h_{\min}\leq h_*$, then $\hot\leq h_*$. 
Thus, by Lemma \ref{lem:newton}, Newton's method
not only is guaranteed to converge, but it also guarantees positivity of all $h^{[k]}$.
}
For this reason, we refer to the algorithm
given by \eqref{eq:IG correction Newton} combined with \eqref{eq:newton} as \emph{positive Newton}.

\subsubsection{Two-step Newton method} \label{sec:modified-newton}
This method, introduced by  McDougall \& Wotherspoon \cite{mcdougall2014simple}, consists of a two-step
predictor-corrector iteration that yields a convergence rate of $1+\sqrt{2}$,
with the same number of function evaluations per iteration as the Newton-Raphson method
presented above. Its first iteration is given by
\begin{equation}
\begin{aligned}
&x^{[\frac12]}=x^{[0]}, \qquad
&x^{[1]}=x^{[0]}-\frac{\psi\left(x^{[0]}\right)}{\psi^{\prime}\left(\frac{1}{2}\left[x^{[0]}+x^{[\frac12]}\right]\right)}=x^{[0]}-\frac{\psi\left(x^{[0]}\right)}{\psi^{\prime}\left(x^{[0]}\right)},
\end{aligned}
\end{equation}
while for $k>1$ we have
\begin{equation}
\begin{aligned}
&x^{[k+\frac12]}=x^{[k]}-\frac{\psi\left(x^{[k]}\right)}{\psi^{\prime}\left(\frac{1}{2}\left[x^{[k-1]}+x^{[k-\frac12]}\right]\right)}, \\
&x^{[k+1]}=x^{[k]}-\frac{\psi\left(x^{[k]}\right)}{\psi^{\prime}\left(\frac{1}{2}\left[x^{[k]}+x^{[k+\frac12]}\right]\right)}.
\end{aligned}
\end{equation}

\subsubsection{Ostrowski's method}
Ostrowski's method is cubically convergent and given by the two-step iteration \revX{\cite{petkovic2014multipoint}}
\begin{subequations} \label{eq: Ostrowski step}
\begin{align}
x^{[k+\frac12]}&=x^{[k]}-\frac{\psi(x^{[k]})}{\psi'(x^{[k]})},\\
x^{[k+1]}&= x^{[k+\frac12]}-\frac{\psi\left(x^{[k+\frac12]}\right)}{\psi'(x^{[k]})}\cdot \frac{\psi(x^{[k]})}{\psi(x^{[k]})-2\psi\left(x^{[k+\frac12]}\right)}.
\end{align}
\end{subequations}


Although higher-order generalizations of Ostrowski's method have been proposed, the original version \eqref{eq: Ostrowski step} has proved to be the most efficient in several contexts. See, e.g \cite{petkovic2014multipoint}, \cite{king1973family}.

\subsubsection{Ostrowski-Newton method}
An efficient combination of Ostrowki's and Newton's methods can be obtained
by performing a single Ostrowski iteration and then proceeding with Newton
iterations (if the convergence criterion is not satisfied already).
The Newton iteration is thus initialized using
$\hot$ as defined in \eqref{eq:IG correction Newton}, but where $\ho$ is
the result of one iteration of Ostrowski's method.

\revY{
        \subsubsection{Bounding quadratic polynomials} \label{subsubsec: Quadratic polynomials}
}  
This iterative method is based on an algorithm proposed by Guermond \& Popov \cite{guermond2016fast} in the context of the Euler equations, but is adapted in a straightforward way to the
shallow water equations.
Assuming we are provided two estimates of $x_*$, $\left(x_{-}^{[0]}, x_{+}^{[0]}\right)$, which are also (respectively) lower and upper bounds, we consider  the Hermite polynomials  
\begin{subequations}
\begin{align}
  P_-(x) &= \psi(\xk_-)+\psi[\xk_-,\xk_-](x-\xk_-)+\psi[\xk_-,\xk_-,\xk_+](x-\xk_-)^2, \label{eq: quadratic polynomial below}\\
  P_+(x) &= \psi(\xk_+)+\psi[\xk_+,\xk_+](x-\xk_+)+\psi[\xk_-,\xk_+,\xk_+](x-\xk_+)^2.
\end{align}
\end{subequations}
Here 
\begin{align*}
  \psi[\xk_-, \xk_-] &= \psi'(\xk_-), \quad
  \psi[\xk_-, \xk_-, \xk_+] &= \frac{\psi[\xk_-,\xk_+]-\psi[\xk_-,\xk_-]}{\xk_+-\xk_-},\quad
  \psi[\xk_-, \xk_+] &= \frac{\psi(\xk_+)-\psi(\xk_-)}{\xk_+-\xk_-}
\end{align*}
are divided differences. 
It can be shown that $P_-(x) \le \psi(x) \le P_+(x)$ for $x\in[x_{-}^k,x_{+}^k]$. 
Both $P_-$ and $P_+$ interpolate $\psi(x)$ at $x_{-}^k$ and $x_{+}^k$ and both 
have one root in the interval $[x_{-}^k,x_{+}^k]$. These roots are given by 
\begin{subequations}
\begin{align}
  x^{[k+1]}_\pm &= \xk_\pm - \frac{2\psi(\xk_\pm)}{\sqrt{\psi'(\xk_\pm)^2-4\psi(\xk_\pm)\psi[\xk_-,\xk_\pm,\xk_+]}}.
\end{align}
\end{subequations}
Proceeding analogously to Lemma 4.5 of Guermond \& Popov \cite{guermond2016fast}, we see that
$x^{[k+1]}_-<x_*<x^{[k+1]}_+$. Moreover, the convergence
of both $\xk_-$ and $\xk_+$ is cubic.

As in the case of the positive Newton method, if $h_{\min}\leq h_*$, we can use $h_{-}^{[0]}:=h_{\min}$ as initial lower bound.
 Furthermore Lemma \ref{lemma:upper bound two rarefaction} ensures that  $h_{+}^{[0]}:=h_\RR$ is always an upper bound for $h_*$. 

\begin{remark}
We have designed and implemented a similar approach based on cubic polynomials and yielding quartic
convergence.  However, due to the difficulty of solving the cubic, that method was
less efficient and robust than other methods considered here, so we omit its description.
\end{remark}

\revY{
        \subsubsection{Single quadratic polynomial} \label{subsubsec: Single quadratic polynomial}

The method of bounding quadratic polynomials was originally introduced to obtain an upper bound
for the intermediate pressure in the Riemann problem for the EEs, such that a prescribed error for the maximum
wave speed is guaranteed. 
However, we are interested in approximating the middle state iteratively until certain convergence criterion, 
that relies uniquely on the current iterate, is met 
(see Section \ref{sec:termination}). 
Thus, the computation of two bounds for $h_*$, $h_{-}^{[k]}$ and $h_{+}^{[k]}$, at each step is not necessary.
Instead, we can use a single quadratic polynomial to approximate $\phi$ from below and whose root, $h_{+}^{[k]}$,
approximates $h_*$ from above.
This polynomial is constructed as in \eqref{eq: quadratic polynomial below}, but the 
lower bound is fixed as $h_{-}^{[k]}=h_{-}^{[0]}$ .

\subsubsection{Single linear polynomial}
\label{subsubsec: Single linear polynomial}
A further simplification of the previous method is to use a single linear polynomial to approximate $\phi$.
This polynomial interpolates $\phi$ at the lower and upper bounds of $h_*$, $h_{-}^{[0]}$ and $h_{+}^{[k]}$;
moreover, its root,
\begin{align}
h^{[k+1]}_{+} = \frac{\phi(h_{+}^{[k]})h_{-}^{[0]}-\phi(h_{-}^{[0]})h_{+}^{[k]}}{\phi(h_{+}^{[k]})-\phi(h_{-}^{[0]})},
\end{align}
is an approximation of $h_*$ from above by the concavity of $\phi$.
}


\subsection{Termination criteria} \label{sec:termination}
An appropriate convergence criterion must be established in order to properly
assess and compare the efficiency of the iterative solvers.
One standard approach is to terminate if the residual is below a specified
tolerance:
\begin{align} \label{residual}
|\phi(h^{[k]})|<\varepsilon_{r}.
\end{align}
This can be checked at no extra cost for the methods considered here, since
the depth function must be evaluated at each step for the next approximation.
Note that condition \eqref{residual} implies the rough bound
\begin{align}\label{eq:best bound residual}
|\hk-h_*| \lessapprox \frac{\varepsilon_r}{|\phi'(h_*)|},
\end{align}
which is practically the best that one can hope for based on the conditioning
of the problem.  In particular, in floating-point arithmetic, the best
possible solution is accurate to approximately $\varepsilon_m/\phi'(h_*)$, where
$\varepsilon_m$ is the unit roundoff.  
On the other hand, \eqref{eq:best bound residual} implies that the requirement \eqref{residual}
 can be too difficult to satisfy when $|\phi'(h_*)|\gg 1$, as discussed by Quarteroni et al. \cite[Ch. 6.5]{quarteroni2010numerical}.

Another widely used convergence criterion is based on stagnation, i.e.,
stopping when the relative change in the solution from one iteration to the
next is below a specified tolerance, as is done by Toro \cite{toro2001shock}:
\begin{equation} \label{stagnation}
\frac{|h^{[k]}-h^{[k-1]}|}{\frac12 |h^{[k]}+h^{[k-1]}|}<\varepsilon_s.
\end{equation}
From the estimate \cite{quarteroni2010numerical} 
\begin{align}
|h^{[k-1]}-h_*|\approx \left|\frac{\hk-h^{[k-1]}}{1-\phi'(h_*)}\right|,
\end{align}
we obtain that the bound \eqref{stagnation} implies
\begin{align}
|h^{[k-1]}-h_*| \lessapprox
\frac12 \left| h^{[k]}+h^{[k-1]}\right|  \frac{\varepsilon_s}{|1-\phi'(h_*)|}.
\end{align}
It is clear that when $|\phi'(h_*)-1|\ll 1$ or $\gg 1$, the stagnation criterion
leads to the same difficulties faced with the residual criterion. 
To obtain a bound on the error of $h^{[k-1]}$, the computation of an
extra iterate is required if we desire to use stagnation.
For this reason, the results presented in this work were obtained with the residual
convergence criterion \eqref{residual}.
There exist modifications of the residual criterion, like \cite{kelley1995iterative} 
 \begin{align} \label{eq:modified_residual}
|\phi(h^{[k]})|<|\phi(\ho)|\varepsilon_{r1}+\varepsilon_{r2},
\end{align}
that ensure (for a fixed initial guess strategy) that the method be invariant under scaling of the problem.
Since we want to compare the influence of different initial guesses in the performance
of the iterative solvers, we prefer to use \eqref{residual} in order to avoid a bias that
would favor an inaccurate initial guess.
 
\revY{
\begin{remark}
  \label{rem:abs_error}
The bracketing property of the bounding quadratic polynomials from Section \ref{subsubsec: Quadratic polynomials}
, $h_{-}^{[k]}\leq h_* \leq h_{+}^{[k]}$, allows the use of an absolute error measure as a termination criterion,
$
h^{[k]}_{+}-h_*\leq h^{[k]}_{+}-h^{[k]}_{-} < \varepsilon_a
$,
instead of a relative one.
This is particularly useful when $|\phi'(h_*)|\gg 1$ and the residual termination criterion becomes hard to satisfy.
However, out of all the methods considered in this work, the absolute error criterion is 
only suitable for the method of bounding quadratic polynomials.
\end{remark}
}

 
    \subsection{Performance comparison}
    \label{sec:Results SWEs}
 
 In this section we study how quickly each method reaches a desired solution
 accuracy, in wall clock time. In
 order to do so, we run each method on a fixed set of $10^7$  Riemann
 problems generated pseudo-randomly\footnote{using the Mersenne Twister algorithm
 provided by the \textit{random} library in Python 3.9.7.}. Following the approach
 used by Gottlieb \& Groth \cite{gottlieb1988assessment}, 20\% of the problems
 involve strong waves and 80\% involve only weak waves. For the strong wave problems, the initial depth
 on each side is taken as $h=10^k$ with $k$ distributed uniformly in $[-4,4]$ \revX{$\left(k\sim \mathcal{U}([-4,4])\right)$};
the velocities
 are taken as $u_{\ell}= 10^k$ and $u_r=-10^k$, with \revX{$k\sim \mathcal{U}([-2,2])$}. 
 For the weak waves the depth is taken as \revX{$h_\ell, h_r \sim \mathcal{U}([0.1,1])$}, and velocities are
 set to zero.  The rationale behind this approach is that in numerical applications most Riemann
 problems that arise will involve relatively nearby states.

 The Riemann solvers
 were implemented in Fortran 90 , the sets of Riemann problems were provided
 to the modules compiled with f2py 1.22.3 and GNU Fortran compiler 11.2.0. 
 Similar tests were performed using Intel's Fortran compiler and, noting that the proportion of execution times between different methods was preserved,
 for the sake of reproducibility we present the results obtained with the GNU compiler.
 
\subsubsection{Comparison of initial guesses}
To evaluate the various initial guesses on an equal footing, we test each of them as a starting
point for the positive Newton method.
As a benchmark for accuracy we use the average relative initial error  (ARIE) over the set of problems,
defined as $\frac{1}{10^7}\sum_{i=1}^{10^7}|{h_*}_i-{\ho}_i|/|{h_*}_i|$.

In our implementation, we always check for a two-rarefaction solution and
use the exact solution \eqref{eq:IG RR} in that case.
This step ensures that $h_{\min} \le h_*$ which is
crucial to avoid negative depth with the positive Newton, Ostrowski-Newton, and bounding polynomials methods.
For the initial guesses that make use of the lower bound, we can obtain improved accuracy
by using $h_{\max}$ if $\phi(h_{max})<0$.
It can be shown that $h_{\max}$ is a lower bound for $h_*$ in this case because
of the monotonicity of the depth function and the
entropy conditions for the shallow water equations; the same estimate is used in
\cite{guermond2016fast} for the Euler equations.
Since the two-rarefaction solution can be checked for \emph{a priori} and written 
explicitly, we
study only problems for which the solution involves at least one shock.

The results are presented in Table \ref{tab:time newton solver with all IG SWE}.
In terms of speed, results for all methods are within about 20\% of each other.
We see that the initial guess achieving the fewest average iterations is the
CC initial guess \eqref{eq:IG CC}, followed by SS \eqref{eq:IG SS SWEs}. 
The AV initial guess gives a slightly faster run-time because it is
cheaper to initialize and because the strong wave problems (for which it yields
a large initial error and hence requires more iterations) compose only a small fraction of the test set.

The two most accurate initial guesses are CC \eqref{eq:IG CC} for problems with weak waves 
and QA \eqref{eq:IG QA} for problems with strong waves.
Both of these have slower than average overall run-time due to 
requiring three depth function evaluations in the worst case.
Depending on the desired accuracy, this suggests the possibility of using an
approximate state Riemann solver that may not need to be refined through iteration,
as is done in, e.g., \cite[Chap. 10.3]{toro2001shock}, \cite[Chap. 9]{toro2013riemann}.
An adaptive noniterative Riemann solver using the CC and QA
initial guesses is described in Appendix \ref{sec: Approximate state RSs}.

The SS approximation is nearly as efficient as the fastest methods and is
nearly as accurate as the most accurate initial guesses (CC and QA).

\begin{table}
\centering
\begin{tabular}{|l|c|c|cc|}
\hline
\multirow{2}{*}{Initial guess}                 & \multirow{2}{*}{Time} & \multirow{2}{*}{\begin{tabular}[c]{@{}c@{}}Average\\ iterations\end{tabular}} & \multicolumn{2}{l|}{Average relative initial error} \\ \cline{4-5} 
                                               &                       &                                                                               & \multicolumn{1}{c|}{Weak waves}        & Strong waves        \\ \hline
AV (average) \eqref{eq:IG AV}                  & 0.900 s               & 3.0                                                                           & \multicolumn{1}{c|}{$4.75 \%$}         & $3500\%$            \\ \hline
RR (two-rarefaction) \eqref{eq:IG RR}          & 0.906 s               & 2.5                                                                           & \multicolumn{1}{c|}{$0.6 \%$}          & $4200\%$            \\ \hline
QA (quadratic approximation) \eqref{eq:IG QA}  & 1.013 s               & 2.3                                                                           & \multicolumn{1}{c|}{$0.6\%$}           & $1.18\%$            \\ \hline
CC (convex combination) \eqref{eq:IG CC}       & 1.042 s               & 2.0                                                                           & \multicolumn{1}{c|}{$0.08\%$}          & $4.95\%$            \\ \hline
PV (primitive-variables) \eqref{eq:IG PV SWEs} & 0.959 s               & 3.0                                                                           & \multicolumn{1}{c|}{$4.75\%$}          & $4300\%$            \\ \hline
SS (two-shock) \eqref{eq:IG SS SWEs}           & 0.906 s               & 2.2                                                                           & \multicolumn{1}{c|}{$0.12\%$}          & $10.95\%$           \\ \hline
HLLE \eqref{eq:IG HLLE}                        & 1.093 s               & 3.1                                                                           & \multicolumn{1}{c|}{$7.71\%$}          & $4100\%$            \\ \hline
\end{tabular}
\caption{Comparison of execution time, average iterations, and average relative initial error
(before iterating) obtained using
the positive Newton iterative solver for
the shallow water equations with different initial guesses.
\label{tab:time newton solver with all IG SWE}}
\end{table}

\subsubsection{Comparison of root-finding algorithms}
We now proceed to compare the various iterative algorithms.  
Where possible, we use the two-shock approximation initial guess, since it gave the best 
trade-off between initial accuracy and overall execution time in the previous comparison.
We consider two different values for the tolerance: $10^{-6}$ and $10^{-12}$.
The results, presented in Table \ref{tab:time iterative solvers with fastest IG SWEs},
show that all methods\revY{,
 except
 the ones based on interpolating quadratic and linear polynomials (Sections 
 \ref{subsubsec: Quadratic polynomials} - \ref{subsubsec: Single linear polynomial}), }
fall within about
10\% of each other.  The Ostrowski-Newton method is the fastest for both tolerances.
We can attribute this to the fact that Ostrowski's method attains the
convergence criteria in the first half-step for most of the problems, as
reflected in the number of iterations,
while the positive Newton method avoids nonphysical iterates as described in Section \ref{sec:newton}.
We also note that the two-step
Newton method of Section \ref{sec:modified-newton} is slightly faster than
the positive Newton method of Section \ref{sec:newton}.
However, one might still prefer the latter due to its outstanding robustness and provable convergence.

For Riemann problems with extremely large and/or small initial states, two additional
challenges arise.  For the solvers in which positivity is not guaranteed, the depth may
become negative.  For all solvers, it may be impossible to reach a prescribed small tolerance
for the residual because $|\phi'(\hk)|\revY{\gg} 1$, as discussed in Section \ref{sec:termination}.
\revY{
  This can be alleviated in the case of the bounding quadratic polynomials by terminating the iterative
  process when either the residual-based or the absolute error (see Remark \ref{rem:abs_error}) criteria  are satisfied.
}
The positive Newton's and Ostrowski-Newton's methods are remarkably less prone than other solvers
to stagnate before reaching the prescribed tolerance.
Since a very low number of iterations is required for typical Riemann problems, one way to deal with this
issue is to fix the maximum number of
iterations \emph{a priori} as suggested by Toro \cite{toro2013riemann}.


\begin{table}
\centering
\begin{tabular}{ |l|*{4}{c|}  } 
\hline
 \diagbox[dir=NW]{Method}{Tolerance}&\multicolumn{2}{|c|}{ $10^{-6}$}&\multicolumn{2}{c|}{$10^{-12}$}\\\hline
										 	          & Time & Iterations  & Time & Iterations    \\ \hline
Positive Newton-SS 								& 0.732 s & 1.4      & 0.874 s  & 2.2 	   \\\hline
Ostrowski-SS        							& 0.724 s & 1.0      & 0.860 s  & 1.3 				   \\\hline
Ostrowski-Newton-SS          			& 0.668 s & 1.1      & 0.780 s  & 1.4 	   \\\hline
Two-step Newton-SS      					& 0.710 s & 1.4      & 0.867 s  & 2.2 		   \\\hline
Bounding quadratic polynomials-RR & 1.252 s & 2.1      & 1.672 s  & 2.6 				   \\\hline
Single quadratic polynomial-RR    & 0.918 s & 2.1      & 1.237 s  & 2.6 				   \\\hline
Single linear polynomial-RR       & 0.805 s & 2.3      & 1.154 s  & 3.3 				   \\\hline
\end{tabular}
\caption{Comparison of execution time for different iterative Riemann Solvers for shallow water equations with their best-performing initial guess\label{tab:time iterative solvers with fastest IG SWEs}}
\end{table}

%


\section{The Euler equations}\label{sec:Euler}

    \subsection{Exact solution of the Riemann problem}

Just as for the shallow water equations, the Riemann problem for the Euler
equations of an ideal compressible gas in 2D or 3D reduces essentially
to the 1D problem if we consider a planar, coordinate-aligned Riemann problem;
see e.g. \cite[Section 18.8.1]{levequefvmbook}.  In that case we have

%
%
%

%

    \begin{subequations}\label{eq:euler}
\begin{align}
\partial_x\mathbf{q}+\partial_x\mathbf{f}(\mathbf{q})=\mathbf{0}, \\
\mathbf{q}(x, 0)=\left\{\begin{array}{ll}
\mathbf{q}_{\ell} & \text { if } x<0, \\
\mathbf{q}_{r} & \text { if } x>0,
\end{array}\right.
\end{align}
\end{subequations}

with 

\begin{equation} \label{eq:1d euler RP}
\mathbf{q}=\left[\begin{array}{c}
\rho \\
\rho u \\
E
\end{array}\right], \quad \mathbf{f}=\left[\begin{array}{c}
\rho u \\
\rho u^{2}+p \\
u(E+p)
\end{array}\right],
\end{equation}
and 
\begin{equation}
E=\rho\left(\frac{1}{2} u^{2}+e\right), \qquad e=\frac{p}{(\gamma-1)\rho},
\end{equation}
\revX{where the heat capacity ratio $\gamma$ is assumed to be constant. }

To describe the exact solution for \eqref{eq:euler} it is useful to use the
primitive variables $\bfw=(\rho,u,p)^T $. The solution consists of three
waves, each one associated with a characteristic speed of the system, given by
the eigenvalues of the flux Jacobian
$\mathrm{A}(\bfq)=\mathbf{f}'(\mathbf{q})$:
\begin{equation}
\lambda_1=u-a, \qquad \lambda_2=u, \qquad\lambda_3=u+a, 
\end{equation}
where $a=\sqrt{\frac{\gamma p}{\rho}}$ is the sound speed. The 1- and 3- waves
are genuinely nonlinear, and can be either shocks or rarefactions, while the 2-
wave is linearly degenerate and thus a contact discontinuity. They separate
four constant states $\mathbf{w}_\ell, \mathbf{w}_{*\ell}, \mathbf{w}_{*r},
\mathbf{w}_r$, where $\bfw_\ell$ and $\bfw_r$ are given by the initial data and
the intermediate states are given by $\bfw_{*\ell}=(\rho_{*\ell},u_*,p_*)$ and
$\bfw_{*r}=(\rho_{*r},u_*,p_*)$. Thus the unknown quantities
are $p_*,u_*,\rho_{*\ell},\rho_{*r}$. If $p_*<p_\ell$ the 1-wave is a
rarefaction and connects $\bfw_\ell$ and $\bfw_{*\ell}$ through the isentropic
relation and the generalised Riemann invariants. If $p>p_\ell$ the 1-wave is a
shock, so $\bfw_\ell$ $\bfw_{*\ell}$ are related by the
Rankine-Hugoniot conditions. Analogous statements hold for the 3-wave. Taking all of
this into account we obtain that $p_*$ is the root of the pressure function

\begin{equation}
\label{eq:pressure function}
\phi\left(p; \bfw_{\mathrm{\ell}}, \bfw_{\mathrm{r}}\right) = f\left(p; \bfw_{\mathrm{\ell}}\right)+f\left(p; \bfw_{\mathrm{r}}\right)+ u_{\mathrm{r}}-u_{\mathrm{\ell}},
\end{equation}

where $f$ is given by
\begin{equation}
f\left(p; \bfw_\lr\right)=\left\{\begin{array}{ll}
\left(p-p_\lr\right)\left[\frac{C_\lr}{p+B_\lr}\right]^{\frac{1}{2}} & \text { if } p>p_\lr(\text { shock }), \\
\frac{2 C_\lr}{(\gamma-1)}\left[\left(\frac{p}{p_\lr}\right)^{\frac{\gamma-1}{2 \gamma}}-1\right] & \text { if } p \leq p_\lr(\text { rarefaction }),
\end{array}\right.
\end{equation}

and $C_\lr, B_\lr$ are given by
\begin{equation}
\begin{array}{l}
C_\lr=\frac{2}{(\gamma+1) \rho_\lr}, \quad B_\lr=\frac{(\gamma-1)}{(\gamma+1)} p_\lr
\end{array}.
\end{equation}

Let us assume that we already have found $p_*$, and therefore the wave structure for our problem. Then, we have closed-form solutions for $u_*$, $\rho_{*\ell}$, and $\rho_{*r}$.  First, we compute the intermediate velocity
\begin{equation}
u_{*}=\frac{1}{2}\left(u_{\mathrm{\ell}}+u_{\mathrm{r}}\right)+\frac{1}{2}\left[f\left(p_*; \bfw_{\mathrm{\ell}}\right)+f\left(p_*; \bfw_{\mathrm{r}}\right)\right].
\end{equation}

The values of $\rho_{*\ell}$ and $\rho_{*r}$ will be also determined by the nature of the 1- and 3- waves. If the 1-wave or the 3-wave are shocks, then
\begin{equation}\label{eq: density if a wave shock}
\rho_{*\ell}=\rho_{\ell}\left[\frac{\frac{p_{*}}{p_{\mathrm{L}}}+\frac{\gamma-1}{\gamma+1}}{\frac{\gamma-1}{\gamma+1} \frac{p_{*}}{p_{\ell}}+1}\right]\qquad 
\text{or}\qquad 
\rho_{*r}=\rho_{\mathrm{r}}\left[\frac{\frac{p_{*}}{p_{\mathrm{r}}}+\frac{\gamma-1}{\gamma+1}}{\frac{\gamma-1}{\gamma+1} \frac{p_{*}}{p_{\mathrm{r}}}+1}\right],
\end{equation}
and if the 1-wave or the 3-wave are rarefactions, then
\begin{equation} \label{eq: density if a wave rarefaction}
\rho_{*\ell}=\rho_{\mathrm{L}}\left(\frac{p_{*}}{p_{\mathrm{L}}}\right)^{\frac{1}{\gamma}}\qquad 
\text{and} \qquad
\rho_{*r}=\rho_{\mathrm{R}}\left(\frac{p_{*}}{p_{\mathrm{R}}}\right)^{\frac{1}{\gamma}},
\end{equation}
respectively.

For an in-depth derivation of the Euler equations and its solution strategy, we refer the reader to, e.g. \cite{toro2013riemann,levequefvmbook}.
An excellent review of exact solver algorithms can be found in the work of Gottlieb \& Groth \cite{gottlieb1988assessment}.

Since the solution of the Riemann problem \eqref{eq:euler} can be explicitly computed once $p_*$ is computed,
 our goal is to determine the root of the pressure function \eqref{eq:pressure function} in a computationally efficient way.
We remark that in the presence of a vacuum state, i.e. $p_\ell$, $p_r$, or $p_*$ equal to zero,
the complete exact solution for the Riemann problem can be obtained explicitly,
and therefore we do not consider problems with this behavior here. 
\revY{However, for practical purposes, an implementation of the vacuum-state treatment presented in \cite[Sec. 4.6]{toro2013riemann}
is available in the reproducibility repository \cite{repository}.}

\subsubsection{Properties of the pressure function}
The following general properties of the pressure function, which behaves similarly to the depth function for the
SWEs, are useful in designing algorithms for its solution.  We have
\begin{equation}
f'\left(p; \bfw_\lr\right)=\left\{\begin{array}{ll}
\left(\frac{C_{\lr}}{B_{\lr}+p}\right)^{1 / 2}\left[1-\frac{p-p_{\lr}}{2\left(B_{\lr}+p\right)}\right] & \text { if } p>p_{\lr} \text { (shock), } \\
\frac{1}{\rho_{\lr} a_{\lr}}\left(\frac{p}{p_{\lr}}\right)^{-(\gamma+1) / 2 \gamma} & \text { if } p \leq p_{\lr} \text { (rarefaction), }
\end{array}\right.
\end{equation}

and
\begin{equation}
f''\left(p; \bfw_\lr\right)=\left\{\begin{array}{ll}
-\frac{1}{4}\left(\frac{C_{\lr}}{B_{\lr}+p}\right)^{1 / 2}\left[\frac{4 B_{\lr}+3 p+p_{\lr}}{\left(B_{\lr}+p\right)^{2}}\right] & \text { if } p>p_{\lr} \text { (shock), } \\
-\frac{(\gamma+1) a_{\lr}}{2 \gamma^{2} p_{\lr}^{2}}\left(\frac{p}{p_{\lr}}\right)^{-(3 \gamma+1) / 2 \gamma} & \text { if } p \leq p_{\lr} \text { (rarefaction). }
\end{array}\right.
\end{equation}

It is straightforward to see that   $\phi^{\prime}(p)=f'(p,\bfw_\ell)+f^{\prime}(p,\bfw_r)>0$ and  $\phi^{\prime \prime}(p)=f^{\prime \prime}(p,\bfw_\ell)+f^{\prime \prime}(p,\bfw_r)<0$, therefore $\phi$ is monotonically increasing and concave down. Therefore, it follows that $\phi\left(p_{\min}\right)>0$ implies that both waves are rarefactions, and
$\phi\left(p_{\max}\right)<0$ implies that both waves are shocks, where $ p_{\min} =\min \left(p_{\ell}, p_{r}\right)$ and $p_{\max} =\max \left(p_{\ell}, p_{r}\right)$.

\subsection{Computation of the initial guess}
    
 In order to find the root of $\phi$ exactly we follow an iterative process. We will proceed as in Section \ref{subsec: IG}.
 We consider a total of six initial guesses, three of which we describe in the following sections.
 The other three (AV, CC, and HLLE) are the natural analogs of the SW initial guesses described in Sections
 \ref{IG: AV}, \ref{IG: CC} and \ref{IG:HLLE}.
 
\subsubsection{Two-rarefaction approximation (RR)}
If the 1- and 3- waves are rarefactions, then
\begin{equation}
\phi(p;\bfw_\ell,\bfw_r) = \phi_\RR(p) := \frac{2C_\ell}{(\gamma-1)}\left[\left(\frac{p}{p_\ell}\right)^{\frac{\gamma-1}{2\gamma}}-1\right]+\frac{2C_r}{(\gamma-1)}\left[\left(\frac{p}{p_r}\right)^{\frac{\gamma-1}{2\gamma}}-1\right],
\end{equation}
and its root is given explicitly by
\begin{equation} \label{eq:two rarefactions euler}
p_{\mathrm{RR}} := \left[\frac{a_{\mathrm{\ell}}+a_{\mathrm{r}}-\frac{1}{2}(\gamma-1)\left(u_{\mathrm{r}}-u_{\mathrm{\ell}}\right)}{a_{\mathrm{\ell}} / p_{\mathrm{\ell}}^{\frac{\gamma-1}{2 \gamma}}+a_{\mathrm{r}} / p_{\mathrm{r}}^{\frac{\gamma-1}{2 \gamma}}}\right]^{\frac{2 \gamma}{\gamma-1}}.
\end{equation}
For problems whose solution does not include two rarefactions, this value can be used as an initial guess.

\subsubsection{Primitive-variable linearized solver (PV)}
Any approximate Riemann solver can be used to provide an initial guess for
an iterative solver.  Here we consider the use of the linearized solver
introduced by Toro \cite{toro1991}.  Linearizing the system
around the state $\hat{\bfw}$ in primitive variables gives
the linearized system \eqref{eq:linearized system}.
Choosing $\hat{\rho}$ and $\hat{a}$ as the arithmetic mean between the left
and right values, we get the following solution:
$$
p_* = \frac{1}{2}(p_\ell+p_r)-\frac18(u_r-u_\ell)(\rho_\ell+\rho_r)(a_\ell+a_r).
$$
In case this value is negative, it is common practice to take $p_* = \epsilon \ll 1$.
However, since we know that (except in the case of two rarefactions)
$\min(p_\ell,p_r)$ is a positive lower bound for $p_*$, we set
\begin{equation}
\label{eq:PPV estimation Euler}
\po = p_\PV := \max\left(\min(p_\ell,p_r),\frac{1}{2}(p_\ell+p_r)-\frac18(u_r-u_\ell)(\rho_\ell+\rho_r)(a_\ell+a_r)\right).
\end{equation}

\subsubsection{Two-shock approximation (SS)} 
Next we consider another initial guess based on an approximate solver that was
proposed by Toro \cite{toro1995direct}. If both waves are
shocks, then the pressure function becomes
\begin{equation}\label{eq:two-shock euler}
\phi_\SS(p) := \left(p-p_{\ell}\right) g_{\ell}(p)+\left(p-p_{\ell}\right) g_{r}(p)+u_r-u_\ell,
\end{equation}
where
\begin{equation}
g_\lr(p)=\left[\frac{C_\lr}{p+B_\lr}\right]^{1 / 2}.
\end{equation}
This can be solved explicitly if we choose an approximation for the argument of $g_\lr$.  We take the
approximate value $p_\PV$, yielding
\begin{equation}
\phi_\SS(p)=\left(p-p_{\ell}\right) g_{\ell}(p_\PV)+\left(p-p_{\ell}\right) g_{r}(p_\PV)+u_r-u_\ell,
\end{equation}
  from which we get the initial guess
  \begin{equation} \label{SS-Euler}
  \po = p_\SS := \frac{g_{\ell}\left(p_\PV\right) p_{\ell}+g_{r}\left(p_\PV\right) p_{r}-u_r+u_\ell}{g_{\ell}\left(p_\PV\right)+g_{r}\left(p_\PV\right)}.
  \end{equation}

\subsection{Iterative algorithms}
 For the Euler equations, we will test all of the root-finding iterative
 algorithms introduced in Section \ref{subsec: Iterative algorithms},
 along with two more described below.
 These additional solvers were studied by Gottlieb \& Groth \cite{gottlieb1988assessment}
 and shown to perform better than methods introduced before that.
 Besides the algorithms discussed here, we also implemented and tested
 Pike's exact Riemann solver for the Euler equations \cite{pike1993riemann}.
 However, we were not able to reproduce the results from \cite{pike1993riemann} or to obtain
 competitive results with that solver, so we omit it here.

 \subsubsection{Gottlieb \& Groth's method (GG)}
A novel approach to compute the intermediate
state of a Riemann problem exactly was introduced by Gottlieb \& Groth \cite{gottlieb1988assessment}. It is a Newton iteration
in which the intermediate state velocity (rather than the pressure) is updated directly:
\begin{equation}
u^{[k+1]}=u^{[k]}-\frac{p_\ell^*(u^{[k]})-p_r^*(u^{[k]})}{p_\ell^{*'}(u^{[k]})-p_r^{*'}(u^{[k]})}.
\end{equation}
Here the values of $p_l^*$, $p_l^{*'}$, $p_r^*$, $p_r^{*'}$ depend on the
nature of the 1- and 3- waves given that $u^{[k]}$ is the middle state
velocity. For instance, if there is a shock (with $u^{[k]}\leq u_\lr$), 
\begin{align}
p^*_\lr(u^{[k]})&=p_\lr+C_\lr(u^{[k]}-u_\lr)W_\lr,\\
p_\lr^{*'}(u^{[k]})&=\frac{2\, C_\lr W_\lr^3}{1+W_\lr^2},\\
W_\lr&=\frac{\gamma+1}{4}\frac{u^{[k]}-u_\ell}{a_\lr}+\mp\left[1+\left[\frac{\gamma+1}{4}\frac{u^{[k]}-u_\lr}{a_\lr}\right]^2\right]^{1/2}.
\end{align}
Meanwhile, if there is a rarefaction (with $u^{[k]}>u_\lr$), we have
\begin{align}
a^*_\lr&=a_\lr+\mp\frac{\gamma-1}{2}(u^{[k]}-u_\lr),\\
p_\lr^*(u^{[k]})&=p_\lr\left(\frac{a^*_\lr}{a_\lr}\right)^{2\gamma/(\gamma-1)},\\
p^{*'}_\lr(u^{[k]})&=\mp\gamma\,p_\lr^*/a^*_\lr,
\end{align}
where for each $\mp$ we take the minus sign for $\ell$ and the plus sign for $r$.
The authors use a two-rarefaction approximation as initial guess for this method.
Since the iterations are performed on the velocity, the initial guess takes the form
\begin{equation}
u^{[0]}=\frac{(u_\ell+z_1 a_\ell)z_2+(u_r-z_1 a_r)}{1+z_2},
\end{equation}
where
\begin{equation}
z_1=\frac{2}{\gamma-1},\qquad z_2=\frac{a_r}{a_\ell}\left(\frac{p_\ell}{p_r}\right)^{\frac{\gamma-1}{2\gamma}}.
\end{equation}

For this method, we employ a different convergence criterion, following
\cite{gottlieb1988assessment}.
The iteration continues until $|(p^*_r-p^*_\ell)/p^*_r|<\epsilon$, for a
specified tolerance $\epsilon$. This method takes advantage of the fact that
the sound
speed $a$ appears more often than the density $\rho$ in computational
simulations of perfect gases. The goal is not only to compute the constant
primitive variables of the intermediate state $p^*$ and $u^*$, but also to
provide an estimate of $\rho_{*\ell}$ and $\rho_{*r}$ through $a_{*\ell}$ and
$a_{*r}$ in order to completely characterize the waves in the Riemann solution.

\subsubsection{van Leer's method (vL)}
van Leer introduced a modification of Godunov's
second formulation \cite{van1979towards,gottlieb1988assessment} to estimate
$p_*$.
This method consists of reducing the difference between two estimates
$u^*_\ell$ and $u^*_r$ for the intermediate velocity.
The functions $u^*_\lr(\cdot)$ are determined by the wave jump equations and can be represented by 
\begin{equation}
u^*_\lr(p^{[k]})=u_\lr\, \mp \frac{p^{[k]}-p_\ell}{A_\lr(p^{[k]})},
\end{equation}
and their derivatives  by
\begin{equation}
{u^*}'(p^{[k]})=
\begin{cases}
\mp(A_\lr^2+\gamma p_k\rho_k)/(2A_\lr^3)\qquad & p^{[k]}\geq p_\lr,\\
\mp (\gamma p_\lr\rho_\lr)^{-1/2}\left[\frac{p^{[k]}}{p_\lr}\right]^{-(\gamma+1)/(2\gamma)}   &p^{[k]}<p_\lr.
\end{cases}
\end{equation}
where again for each $\mp$ we take the minus sign for $\ell$ and the plus sign for $r$, and
\begin{equation}
A_\lr=\begin{cases}
(\gamma\, p_\lr \rho_\lr)^{1/2}\left[\frac{\gamma+1}{2\gamma}\frac{p^{[k]}}{p_\lr}+\frac{\gamma-1}{2\gamma} \right]^{1/2}\qquad & p^{[k]}\geq p_\lr,\\
\frac{\gamma-1}{2\gamma}(\gamma\,p_\lr\rho_\lr)^{1/2}\frac{1-p^{[k]}/p_\lr}{1-(p^{[k]}/p_\lr)^{(\gamma-1)/(2\gamma)}}  &p^{[k]}<p_\lr.
\end{cases}
\end{equation}
The pressure approximation is updated via
\begin{equation}
p^{[k+1]}=p^{[k]}-\frac{u^*_\ell(p^{[k]})-u^*_r(p^{[k]})}{u^{*'}_\ell(p^{[k]})-u^{*'}_r(p^{[k]})}.
\end{equation}
For this method, we use the convergence criterion $|(u^*_r-u^*_\ell)/u^*_r|<\epsilon$.


    \subsection{Performance comparison}
    \label{sec:results EEs}

For the performance comparison we proceed analogously to the shallow water equations. We consider $10^7$ Riemann problems generated pseudo-randomly, with $20\%$ of them involving strong waves and $80\%$ involving only weak waves.
For the problems with strong waves we take each of the initial pressures as $p_\ell, \, p_r=10^k$ with $k\sim \mathcal{U}([-4,4])$;
the particle  velocities as $u_\ell=10^k$ and $u_r=-10^k$ with $k\sim \mathcal{U}([-2,2])$;
the densities
are taken as \revX{$\rho_\ell, \rho_r \sim \mathcal{U}([0.01,0.9])$}.
For the problems with weak waves we use $p_\ell,\, p_r \sim \mathcal{U}([0.1, 1])$, $u_\ell=u_r=0 $ , and $ \rho_\ell, \rho_r\ \sim \mathcal{U}([0.1,0.9])$. 

Again, in our implementation we always check first for a two-rarefaction solution
and use the exact solution \eqref{eq:two rarefactions euler} in that case.  Therefore, in this section we
consider only Riemann problems whose solution includes one or more shocks.  This is guaranteed for the
strong wave problems since the fluid on both sides is initially flowing toward the origin.
The initial conditions for the weak wave problems were set to resemble the behavior of the Sod shock tube problem \cite{sod1978survey}, 
where the exact Riemann solution consists of a shock, a contact discontinuity, and a rarefaction.
Once the two-rarefaction scenario has been discarded, we have the bound $p_{\min} \leq p_*$ and
the modification of Newton's method \eqref{eq:IG correction Newton} can be used to avoid non-positive pressures.

The solvers of Gottlieb \& Groth and of van Leer were both originally developed
to allow for different ratios of specific heats on the left and right (see \cite{gottlieb1988assessment}, \cite{van1979towards}).
 Notice that both of these solvers compute $u_*$ in the iterative process to obtain $p_*$. Therefore, to allow a fair timing comparison, $u_*$ is computed along with $p_*$ in every solver. 
\revX{\subsubsection{Comparison of initial guesses}}
We compare the initial guesses by applying each in combination with Newton's method, as is done in Section \ref{sec:Results SWEs}.
The results are presented in Table \ref{tab:time newton solver with all IG}.
The two-rarefaction (RR) and convex combination (CC) approximations require fewer iterations on average.
However, the greater cost of the CC approximation itself outweighs the advantage obtained
from the reduced number of iterations required, yielding the slowest method overall.
Again we conclude that the two-shock approximation is the best initial guess
available for an iterative scheme (assuming that one always checks first for the
two-rarefaction case).

In Appendix \ref{sec: Approximate state RSs}
we define an adaptive noniterative Riemann solver based on the CC approximation, similar to that proposed for the SWEs.
This approximation can be used as a solver for problems with weak weaves, and as the initial
guess for an iterative algorithm for problems with strong waves.
Even though for weak wave problems this solver does not provide an exact solution,
the accuracy attained with our ensemble of test problems is
competitive (0.04\% average relative error).
Furthermore, an improvement in execution time of approximately 25\% was obtained
with respect to the fastest fully iterative scheme considered in this work.


\begin{center}
\begin{table}
\centering
\begin{tabular}{|l|c|c|cc|}
\hline
\multirow{2}{*}{Initial guess}                                      & \multirow{2}{*}{Time (s)} & \multirow{2}{*}{Iterations} & \multicolumn{2}{c|}{Average relative initial error} \\ \cline{4-5} 
                                                                    &                           &                             & \multicolumn{1}{c|}{Weak waves}    & Strong waves   \\ \hline
AV (average pressure)  \eqref{eq:IG AV}                             & 2.847                     & 3.6                         & \multicolumn{1}{c|}{$7.05 \%$}     & $51\%$         \\ \hline
RR (two-rarefaction) \eqref{eq:two rarefactions euler}              & 2.893                     & 2.5                         & \multicolumn{1}{c|}{$0.13 \%$}     & $193\%$        \\ \hline
PV (primitive variables linearized) \eqref{eq:PPV estimation Euler} & 2.832                     & 3.5                         & \multicolumn{1}{c|}{$7.05\%$}      & $33\%$         \\ \hline
SS (two-shock) \eqref{SS-Euler}                                     & 2.247                     & 2.3                         & \multicolumn{1}{c|}{$0.11\%$}      & $25\%$         \\ \hline
CC (convex combination)  \eqref{eq:IG CC}                           & 3.317                     & 2.1                         & \multicolumn{1}{c|}{$0.04\%$}      & $28\%$         \\ \hline
HLLE \eqref{eq:IG HLLE}                                             & 2.957                     & 3.3                         & \multicolumn{1}{c|}{$4.35\%$}      & $20\%$         \\ \hline
\end{tabular}
\caption{Comparison of execution time of Newton's iterative solver with different initial guesses and 
their average relative initial error the
 Euler equations with different initial guesses \label{tab:time newton solver with all IG}}
\end{table}
\end{center} 
\revX{\subsubsection{Comparison of root-finding algorithms}}
Next we assess the performance of each iterative scheme.  For most of the
schemes, we use the SS (two-shock) approximation, based on its performance
in the comparison above.  As upper
and lower bounds of $p_*$ for the bounding polynomial iterative scheme we use
$p_\RR$ and $p_{\min}$, respectively.
Given its previously-demonstrated superior performance \cite{gottlieb1988assessment}, we use the RR (two rarefaction) approximation for the flow velocity as starting guess for the GG solver.

The results for two different tolerances are shown in Table \ref{tab:time iterative solvers with fastest IG}.
For $\epsilon=10^{-6}$ we obtain again that the positive Newton and two-step
Newton methods converge faster (in wall clock time) than the other iterative
schemes.  The bounding polynomial-based iterative scheme converges in the
fewest iterations, but is drastically slower due to the cost per iteration.

For $\epsilon = 10^{-12}$, van Leer's method \revY{ and the single linear polynomial-based } method fail to converge for some problems,
either because a negative pressure is generated or a fixed maximum number of iterations is exceeded.
Also, in contrast with the results obtained for the shallow water equations, we
have that \revY{Ostrowski's and Ostrowski-Newton's methods are  faster than
positive Newton's method } at this tighter tolerance.
This suggests that as the function $\phi$ increases in computational complexity
and more accuracy is needed, higher order methods may become more efficient.
Note that Ostrowski-Newton's method preserves the robustness of the positive Newton's method with the fix \eqref{eq:IG correction Newton}.
This property relies on the way the Fortran compiler handles undefined numbers if the first Ostrowski half-step \eqref{eq: Ostrowski step} returns a non-positive approximation; therefore, care must be taken in its implementation.
\begin{center}
\begin{table}
\centering
\begin{tabular}{ |l|*{4}{c|} } 
\hline
 \diagbox[dir=NW]{Method}{Tolerance}&\multicolumn{2}{|c|}{ $\epsilon=10^{-6}$}&\multicolumn{2}{c|}{$\epsilon=10^{-12}$}\\\hline
 											 	& Time      & Iterations& Time      & Iterations    \\\hline
Positive Newton-SS 								& 1.654 s   & 1.5       & 2.258 s   & 2.3 	        \\\hline
Ostrowski-Newton-SS 					        & 1.838 s   & 1.1       & 2.003 s   & 1.4 	        \\\hline
Ostrowski-SS 					                & 1.762 s   & 1.0       & 2.061 s   & 1.3 	        \\\hline
GG-RR		    	                            & 2.860 s   & 2.4       & 3.702 s   & 3.3 	        \\\hline
Two-step Newton-SS						        & 1.670 s   & 1.5       & 2.232 s   & 2.2 		    \\\hline
van Leer-SS								        & 1.973 s   & 2.5       & -         & - 			\\\hline
Bounding quadratic polynomials-RR			    & 4.898 s 	& 1.0       & 5.541 s   & 1.2 			\\\hline
Single quadratic polynomial-RR                  & 3.203 s 	& 1.1       & 3.753 s   & 1.5 			\\\hline
Single linear polynomial-RR			            & -      	& -         & -         & - 			\\\hline

\end{tabular}
\caption{Comparison of execution time for different iterative Riemann Solvers for Euler Equations with their best performing initial guesses \label{tab:time iterative solvers with fastest IG}}
\end{table}
\end{center}


\section{Finite volume simulations}\label{sec:FV}

Riemann solvers are most widely used as a building block in numerical simulations, and we
now compare the performance of the iterative solvers studied in this work with that of
the most commonly-used approximate solvers. 
We use the wave-propagation approach implemented in Clawpack \cite{ketcheson2012pyclaw,mandli2016clawpack}.
The Riemann solvers compute the full characteristic structure of the Riemann
solution in order to obtain the value along the cell interface.
We test the solvers in combination with the first-order Godunov method and
then with the standard second-order Clawpack (Lax-Wendroff-LeVeque) approach.

We first consider a shallow water shock-wave interaction problem inspired by
the Woodward-Collela blast wave problem, with reflecting boundary conditions
at $x=\pm 5$ and the following initial data:
\begin{align}
    h(x)=\begin{cases}
        30, & \text{if }x\leq-2,\\
        1, & \text{if } -2<x<2,\\
        50, & \text{if } 2\leq x,
    \end{cases}
    \qquad
    hu(x)=0.
\end{align}
We compare the accuracy of the solution obtained with the Ostrowski-Newton iterative solver
with the Roe and HLLE approximate solvers, at final time $T=10$ on a range of mesh
sizes, as shown in Table \ref{tab: FV SWEs accuracy}. 
For this we use a self-convergence test with the solution given by each solver in a grid
of 4050 cells as exact solution, and the errors are measured using the corresponding standard grid function norms.
The execution time of each simulation is measured with the same final time and a grid
of 4050 volume cells. 
An important factor for accuracy and timing measurement is the choice of the time step,
which is chosen to yield a CFL number of approximately 0.9 at each step.
The results can be seen in Table \ref{tab: FV SWEs time}, where a grid function
norm is used to compute the relative $L^2$ error in $h$.

For the Euler equations we use the positive Newton iterative solver. We take the domain $[0,1]$, wall boundary conditions, and the initial conditions
\begin{align}
    \rho(x)=0.1,\qquad
    \rho u(x)=0, \qquad
    E(x)=\begin{cases}
        1000/(\gamma-1), & \text{if }x<0.1,\\
        1, & \text{if } 0.1\leq x\leq 0.9,\\
        100/(\gamma-1), & \text{if } 0.9< x.
    \end{cases}
\end{align}
Accuracy and execution time are measured analogously to the shallow water equations with a final time \revX{$T=0.5$}, and the errors being computed in $\rho$.
The results are shown in Tables \ref{tab:FV EEs accuracy} and \ref{tab:FV EEs time}.

We obtain similar results for both the shallow water and Euler equations. For
both problems, all three Riemann solvers give similar accuracy.
We see, on the other hand, that with Godunov's method the approximate solvers are faster than the iterative solver by a factor of 2 for the shallow water equations and for a factor of 3 for the Euler equations. 
These relative differences are somewhat smaller when a second order scheme is used.
For higher-order methods, with more costly additional operations such as WENO reconstruction, ADER schemes,
or for multiphysics simulations, the relative increase in cost will be further diminished.

\begin{table}
\begin{tabular}{cc}
    \begin{minipage}{.5\linewidth}
        \centering
        \begin{tabular}{|c|cccc|cccc|}
\hline
\multirow{2}{*}{Grid} & \multicolumn{4}{c|}{$L^2$ error for $h$ (\%)}                            \\ \cline{2-5} 
                      & \multicolumn{1}{c|}{50}    & \multicolumn{1}{c|}{150}   & \multicolumn{1}{c|}{450}   & 1350     \\ \hline
Iterative             & \multicolumn{1}{c|}{2.07} & \multicolumn{1}{c|}{1.61} & \multicolumn{1}{c|}{1.15} & 0.59 \\ \hline
HLLE                  & \multicolumn{1}{c|}{1.96} & \multicolumn{1}{c|}{1.48} & \multicolumn{1}{c|}{1.17} & 0.54 \\ \hline
Roe                   & \multicolumn{1}{c|}{1.97} & \multicolumn{1}{c|}{1.52} & \multicolumn{1}{c|}{1.19}  & 0.55 \\ \hline
\end{tabular}
           \caption{Accuracy of Godunov's method for the shallow water equations with different Riemann solvers\label{tab: FV SWEs accuracy}}
        \end{minipage} &

     \begin{minipage}{.5\linewidth}
        \centering
        \begin{tabular}{|c|cccc|cccc|}
\hline
\multirow{2}{*}{Grid} & \multicolumn{4}{c|}{$L^2$ error for $\rho$ (\%)}                                                               \\ \cline{2-5} 
                      & \multicolumn{1}{c|}{50}    & \multicolumn{1}{c|}{150}   & \multicolumn{1}{c|}{450}   & 1350  \\ \hline
Iterative             & \multicolumn{1}{c|}{42.07} & \multicolumn{1}{c|}{34.96} & \multicolumn{1}{c|}{20.24} & 9.64\\ \hline
HLLE                  & \multicolumn{1}{c|}{42.84} & \multicolumn{1}{c|}{31.86} & \multicolumn{1}{c|}{23.51} & 12.3 \\ \hline
Roe                   & \multicolumn{1}{c|}{41.8} & \multicolumn{1}{c|}{34.84} & \multicolumn{1}{c|}{20.22} & 9.65  \\ \hline
\end{tabular}
    \caption{Accuracy of Godunov's method for the Euler equations with different Riemann solvers\label{tab:FV EEs accuracy}}
    \end{minipage}
\end{tabular}
\end{table}

\begin{table}
\begin{tabular}{cc}
    \begin{minipage}{.5\linewidth}
        \centering
        \begin{tabular}{|c|cc|}
        \hline
        \multirow{2}{*}{Riemann Solver} & \multicolumn{2}{c|}{Execution time}             \\ \cline{2-3} 
                                        & \multicolumn{1}{c|}{First order} & Second order \\ \hline
        Iterative                       & \multicolumn{1}{c|}{4.4 s}       & 4.9 s        \\ \hline
        HLLE                            & \multicolumn{1}{c|}{2.3 s}       & 3.3 s        \\ \hline
        Roe                             & \multicolumn{1}{c|}{2.2 s}       & 3.3 s        \\ \hline
        \end{tabular}
       \caption{Execution time of finite volume simulations with different Riemann solvers
       for the shallow water equations\label{tab: FV SWEs time}}
        \end{minipage} &

     \begin{minipage}{.5\linewidth}
        \centering
        \begin{tabular}{|c|cc|}
        \hline
        \multirow{2}{*}{Riemann Solver} & \multicolumn{2}{c|}{Execution time}             \\ \cline{2-3} 
                                        & \multicolumn{1}{c|}{First order} & Second order \\ \hline
        Iterative                       & \multicolumn{1}{c|}{11.0 s}      & 12.3 s       \\ \hline
        HLLE                            & \multicolumn{1}{c|}{3.3 s}       & 5.0 s        \\ \hline
        Roe                             & \multicolumn{1}{c|}{3.6 s}       & 5.2 s        \\ \hline
        \end{tabular}
        \caption{Execution time of finite volume simulations with different Riemann solvers
       for the Euler equations\label{tab:FV EEs time}}
    \end{minipage}
\end{tabular}
\end{table}

\section{Conclusion}
The new positivity-ensuring modifications of iterative solvers presented here, combined with
accurate lower bounds for the depth and pressure, yield iterative solvers that are as robust
as any approximate Riemann solver and are guaranteed to give physically-admissible outputs.
Unsurprisingly, in terms of speed the best iterative methods still lag behind approximate solvers, but by a
smaller margin (about a factor of two) than what has been reported in the past.  This may be
due to relative shifts in the speed of flops compared to memory accesses over recent decades.

The new \revY{cubically}-convergent (bounding \revY{quadratic } 
polynomials) iteration succeeds in reducing the number
of iterations required for solution of the Euler Riemann problem
\revY{and can be coupled with an absolute error-based convergence criterion}, but is overall
inefficient due to the cost per iteration.
\revY{Its simplified versions, relying on the computation of a single interpolating polynomial,
are somewhat more competitive in terms of efficiency, but also less robust.
Furthermore, they are not well-suited for an absolute error-based termination criterion, 
as the lower bound is fixed from the beginning of the iterative process. }  
Among the iterative methods tested, we recommend
the positive Newton method (proposed herein) as a simple, robust, and near-best-efficiency solver.

\section*{Acknowledgements}
We thank the anonymous referees for their helpful comments,
one of which led to the introduction and testing of the single quadratic polynomial method.

\appendix

\section{Approximate state Riemann solvers}
\label{sec: Approximate state RSs}
Given the high accuracy obtained with the CC and QA initial guesses
for the SWE Riemann problem, a new non-iterative approximate state Riemann 
solver  is proposed in Algorithm \ref{algo: adaptive noniterative SWEs}.
We take advantage of the fact that the criterion for determining the nature of the waves
can be reused to compute the CC approximation.
To also obtain an approximation of $u_*$, \eqref{eq: speed star SWEs} can still be used.

\begin{algorithm}
\caption{Adaptive non-iterative Riemann solver for the shallow water equations}\label{algo: adaptive noniterative SWEs}
\begin{algorithmic}[1]
    \Require $\bfql, \bfqr$
    \Ensure $\tilde{h}_*$
    \Comment Approximation of $h_*$
    \State $\hmin\xleftarrow{} \min(h_\ell,h_r)$
    \State {\bf if} $\phi(\hmin)>0$ {\bf then}
    \Comment{Check if solution has two rarefactions }
    \State $\qquad \tilde{h}_*\xleftarrow{}h_{RR}$
    \Comment{Exact solution}
    \State {\bf else} 
    \State $\qquad \hmax\xleftarrow{} \max(h_\ell,h_r)$
    \State $\qquad$ {\bf if} $\phi(\hmax)<0$ {\bf then}
    \Comment{Two shocks (Strong waves)}
    \State $\qquad \qquad \tilde{h}_*\xleftarrow{}h_{QA}$
    \Comment{Quadratic approximation}
    \State $\qquad {\bf else}$
    \Comment{One rarefaction and one shock (Weak waves)}
    \State $\qquad \qquad \tilde{h}_*\xleftarrow{}h_{CC}$
    \Comment{Convex combination}
    \State $\qquad$ {\bf end if}
    \State {\bf end if}
\end{algorithmic}
\end{algorithm}

Similarly, we also propose an adaptive solver for the EEs in \ref{algo: adaptive iterative EEs}.
However, as discussed in Section \ref{sec:results EEs}, the CC approximation of $p_*$ might not be accurate enough
when the Riemann solution involves strong waves. Therefore, this proposed adaptive solver employs iteration for
two-shock solutions.

\begin{algorithm}
\caption{Adaptive iterative Riemann solver for the Euler equations}\label{algo: adaptive iterative EEs}
\begin{algorithmic}[1]
    \Require $\bfql, \bfqr$
    \Ensure $\tilde{p}_*$
    \Comment Approximation of $p_*$
    \State $\pmin\xleftarrow{} \min(p_\ell,p_r)$
    \State {\bf if} $\phi(\pmin)>0$ {\bf then}
    \Comment{Check if solution has two rarefactions }
    \State $\qquad \tilde{p}_*\xleftarrow{}p_{RR}$
    \Comment{Exact solution}
    \State {\bf else} 
    \State $\qquad \pmax\xleftarrow{} \max(p_\ell,p_r)$
    \State $\qquad$Compute $\phi(\pmax)$
    \State $\qquad p^{[0]}\xleftarrow{} p_{CC}$
    \State $\qquad$ {\bf if} $\phi(\pmax)>0$ {\bf then}
    \Comment{One rarefaction and one shock (Weak waves)}
    \State $\qquad \qquad \tilde{p}_*\xleftarrow{}p^{[0]}$
    \Comment{Use convex combination as solution}
    \State $\qquad {\bf else}$
    \Comment{Two shocks (Strong waves)}
    \State $\qquad \qquad p^{[0]}\xleftarrow{}\max(\pmax,p^{[0]}-\phi(p^{[0]})/\phi'(p^{[0]}))$
    \Comment More accurate Positive Newton modification
    \State $\qquad \qquad\tilde{p}_*\xleftarrow{} \text{Newton}(p^{[0]})$
    \Comment{Newton iterations}
    \State $\qquad$ {\bf end if}
    \State {\bf end if}
\end{algorithmic}
\end{algorithm}

\bibliographystyle{plain}                                                                                                                                                                               
\bibliography{bib.bib}

\end{document}